\documentclass[11pt,a4paper,reqno]{amsart} 
\usepackage{amssymb}
\usepackage{latexsym}
\usepackage{amsmath}
\usepackage{mathrsfs}
\usepackage{nccmath}    
\usepackage[T1]{fontenc} 
\usepackage{color}
\usepackage{cite}
\usepackage{color}
\usepackage{comment}
\usepackage{hyperref}
\hypersetup{hidelinks}

\usepackage{calc}                   

\newcommand{\Char}{\operatorname{Char}}
\newcommand{\scal}[2]{\langle #1,#2\rangle}

\newcommand{\rr}[1]{\mathbf R^{#1}}

\newcommand{\nn}[1]{\mathbf N^{#1}}

\newcommand{\nm}[2]{\Vert #1\Vert _{#2}}

\newcommand{\op}{\operatorname{Op}}

\newcommand{\sets}[2]{\{ \, #1\, ;\, #2\, \} }
\newcommand{\Sets}[2]{\left \{ \, #1\, ;\, #2\, \right \} }

\newcommand{\ep}{\varepsilon}
\newcommand{\fy}{\varphi}

\newcommand{\eabs}[1]{\langle #1\rangle}     

\newcommand{\vrum}{\vspace{0.1cm}}

\newcommand{\Op}{\operatorname{Op}}

\newcommand{\GL}{\mathbf{M}}

\newcommand{\maclK}{\mathcal K}

\newcommand{\mascF}{\mathscr F}

\newcommand{\mascS}{\mathscr S}

\newcommand{\phf}{\mathfrak{P}_r^\mathrm{hom}}

\newcommand{\Hphi}{H^\Phi_{s,\sigma}}
\newcommand{\WF}{{\mathrm{WF}}}
\newcommand{\lossthreshold}{\mathfrak{L}_{\Phi,d}}
\newcommand{\WFHphi}{\WF^\Phi_{s,\sigma}}

\numberwithin{equation}{section}          

\newtheorem{thm}{Theorem}
\numberwithin{thm}{section}

\newcommand{\rubrik}{}
\newtheorem{prop}[thm]{Proposition}
\newtheorem{cor}[thm]{Corollary}

\theoremstyle{definition}

\newtheorem{defn}[thm]{Definition}

\theoremstyle{remark}
\newtheorem{rem}[thm]{Remark}

\title{Fourier type operators
on Orlicz spaces and the role of
Orlicz Lebesgue exponents}

\author{Matteo Bonino}

\address{Dipartimento di Matematica ``G. Peano'', Universit\'a
degli Studi di Torino}

\email{matteo.bonino@unito.it}

\author{Sandro Coriasco}

\address{Dipartimento di Matematica ``G. Peano'', Universit\'a
degli Studi di Torino}

\email{sandro.coriasco@unito.it}

\author{Albin Petersson}

\address{Department of Mathematics,
Linn{\ae}us University, Sweden}

\email{albin.petersson@lnu.se}

\author{Joachim Toft}

\address{Department of Mathematics,
Linn{\ae}us University, Sweden}

\email{joachim.toft@lnu.se}

\begin{document}

\begin{abstract}
We deduce continuity and (global) wave-front
properties of classes of Fourier multipliers,
pseudo-differential, and Fourier integral operators
when acting on Orlicz spaces, or more generally,
on Orlicz-Sobolev type spaces. In particular,
we extend H{\"o}rmander's improvement
of Mihlin's Fourier multiplier theorem
to the framework of Orlicz spaces. We also
show how Young functions $\Phi$ of the
Orlicz spaces are linked to properties of
certain Lebesgue exponents $p_\Phi$ and
$q_\Phi$ emerged from $\Phi$.

\end{abstract}



\maketitle

\section{Introduction}\label{sec0}
\par




Orlicz spaces, introduced by W. Orlicz 
in 1932 \cite{Orl}, are Banach spaces 
which generalize the normal $L^p$ spaces 
(see Section \ref{sec1} for notations). 
Orlicz spaces are denoted by $L^\Phi$, 
where $\Phi$ is a Young function, and we 
obtain the usual $L^p$ spaces, 
$1\leqslant p < \infty$, by choosing 
$\Phi(t) = t^p$. For more facts on
Orlicz spaces, see \cite{Rao}.

\par

An advantage of Orlicz spaces is that 
they are suitable when solving certain 
problems where $L^p$ spaces are 
insufficient. As an example, consider 
the entropy of a probability density 
function $f$ given by
$$
E(f) = -\int f(\xi)
\log f(\xi) \, d \xi.
$$
In this case, it may be more suitable to 
work with an Orlicz norm estimate, for 
instance with $\Phi(t) = t \log(1+t)$, 
as opposed to $L^1$ norm estimates. 

\par

The literature on Orlicz spaces is
rich, see e.g. \cite{ApKaZa,
MajLab1,MajLab2,HaH,Mil,OsaOzt}
and the references therein. Recent 
investigations also put pseudo-differential operators in the framework 
of Orlicz 
modulation spaces (cf \cite{TofUst}, see 
also \cite{SchF,ToUsNaOz} for further 
properties on Orlicz modulation spaces).
In this paper, we deal with pseudo-differential operators as well as 
Fourier multipliers in Orlicz spaces.

\par

Results pertaining to continuity 
properties on $L^p$-spaces are
well-established. Our approach is to utilize 
a Marcinkiewicz interpolation-type 
theorem by Liu and Wang in \cite{Liu}
to extend such continuity properties to 
also hold on Orlicz spaces. As an 
initial example, the methods described 
in the subsequent sections allow us to 
obtain the following extension of 
Mihlin's Fourier multiplier theorem (see 
\cite{Mic} for the original theorem).
\begin{thm} 
Let $\Phi$ be a strict Young function 
and $a\in L^\infty (\rr d \setminus
\{ 0\})$ be such that
$$
\sup_{\xi\neq 0}
\left(
|\xi|^{|\alpha|} 
\left|
\partial^\alpha a(\xi)
\right| \right)
$$
is finite for every $\alpha\in \nn d$ 
with $|\alpha|\leqslant
[\frac{d}{2}]+1$. Then $a(D)$ is 
continuous on $L^\Phi(\rr d)$.
\end{thm}

In fact, we also obtain H\"ormander's 
improvement of Mihlin's Fourier 
multiplier theorem (cf \cite{Ho0}) in 
the context of Orlicz spaces. This 
result can be found 
in the first part of
Section \ref{sec3} 
(Theorem \ref{Thm:OrlContHormMult}). In 
a similar manner, we obtain continuity 
results for pseudo-differential 
operators of order $0$ in Orlicz spaces 
as well, see Theorem 
\ref{Thm:OrlContPseudo}. 
This is an alternative proof to the one which can 
%
be achieved by a general analysis of
operators on Banach function spaces,
see \cite{Ka14}.
We further extend this set of results showing the continuity on Orlicz 
spaces of a broad class of Fourier integral operators, under
a condition on the order of the amplitude (that is, a loss of
derivatives and decay), see Theorem \ref{Thm:OrlContSGFIOs}.

The second part of Section \ref{sec3} is devoted to introducing a 
version of (global) wave-front set involving classes of Sobolev spaces 
modelled on Orlicz spaces, and studying its propagation from data to 
solutions of some equations.
\par

Section \ref{sec1} 
also includes 
investigations of Lebesgue exponents 
$p_\Phi$ and $q_\Phi$ constructed from 
the Young function $\Phi$, which are 
important for the interpolation theorem. 
These parameters were described in 
\cite{Liu}, where it was claimed that 
\begin{align}
    p_\Phi < \infty &\iff \Phi \text{ 
    fulfills  the $\Delta_2$-condition}
    \label{eq:pPhi}
    \intertext{and}
    q_\Phi > 1 &\iff \Phi \text{ is 
    strictly convex.}\label{eq:qPhi}
\end{align}
In Section \ref{sec1}, 
we confirm that 
\eqref{eq:pPhi} is correct, but that 
neither logical implication of 
\eqref{eq:qPhi} is correct. Instead, 
other conditions on $\Phi$ are found 
which characterize $q_\Phi > 1$ (see 
Proposition 
\ref{Prop:YounFuncEquivCond}). At
the same time, we deduce a weaker form
of the equivalence \eqref{eq:qPhi}
and show that if $q_\Phi > 1$, then
there is an equivalent Young function
to $\Phi$ which is strictly convex.
(see Proposition
\ref{Prop:YounFuncEquivCond2}).

\par

\section*{Acknowledgements}
The first and second authors have been partially 
supported by INdAM - GNAMPA Project 
CUP\_E53C22001930001 (Sc. Resp. S. Coriasco).
The first and second authors gratefully
acknowledge also the support by the
Department of Mathematics, Linn{\ae}us 
University, V{\"a}xj{\"o}, Sweden,
during their stays in A.Y. 2022/2023 and A.Y. 
2023/2024, when most of the results presented in 
this paper have been obtained.
The third and fourth authors are supported by 
Vetenskapsr{\aa}det
(Swedish Science Council), within the project
2019-04890.
Thanks are due to E. Nabizadeh-Morsalfard, for useful comments leading
to improvements of the content and style.
The authors are also grateful to the anonymous Referees, for their observations 
and suggestions, aimed at improving the overall quality of the paper.


\section{Preliminaries}\label{sec1}

\par


In this section we recall some
facts on Orlicz spaces and
pseudo-differential operators.
Especially we recall Lebesgue
exponents given in e.{\,}g.
\cite{Liu} and explain some
of their features.

\par

\subsection{Orlicz Spaces}\label{subsec1.1}

\par

In this subsection we provide an overview
of some basic definitions and state some
technical results that will be needed.
First, we recall the definition of weak $L^p$ spaces.

\par

\begin{defn}
Let $p\in (0,\infty ]$. The \emph{weak $L^p$ space}
$wL^p(\mathbf{R}^{d})$ consists of all Lebesgue 
measurable functions $f : \rr d \to \mathbf C$
for which
\begin{equation}
\| f  \|_{wL^p}
\equiv
\sup_{t>0} t \left ( \mu _f(t) \right )^{\frac{1}{p}}
\end{equation}
is finite. Here $\mu _f(t)$ is the Lebesgue
measure of the set
$\sets {x\in \rr d}{|f(x)|>t}$.
\end{defn}

\par

\begin{rem}\label{Rem:WeakLp}
Notice that the $wL^p$-norm is not a true norm, since the
triangular inequality fails. Nevertheless, one has that
$\| f  \|_{wL^p} \leqslant \| f  \|_{L^p}$.
In particular, $ L^p(\rr d)$ is continuously embedded
in $wL^p (\rr d)$.
\end{rem}

\par

Next, we recall some facts concerning
Young functions and Orlicz
spaces. (See \cite{Rao,HaH}.)

\par

\begin{defn}\label{Def:ConvFunc}
A function $\Phi :\mathbf R \rightarrow
\mathbf R \cup \{ \infty\}$ is called \emph{convex} if
\begin{equation*}
\Phi (s_1 t_1+ s_2 t_2)
\leqslant s_1 \Phi (t_1)+s_2\Phi (t_2)
\end{equation*}
when
$s_j,t_j\in \mathbf{R}$
satisfy $s_j \geqslant 0$, $j=1,2$, and
$s_1 + s_2 = 1$.
\end{defn}
We observe that $\Phi$ might not
 be continuous, because we permit
 $\infty$ as function value. For example,
$$
\Phi (t)=
\begin{cases}
  c,&\text{when}\ t \leqslant a
  \\[1ex]
   \infty ,&\text{when}\ t>a
\end{cases}
$$
is convex but discontinuous at $t=a$.

\par

\begin{defn}\label{Def:YoungFunc}
Let $\Phi$ be a function from $[0,\infty)$ to $[0,\infty]$.
Then $\Phi$ is called a \emph{Young function} if
\begin{enumerate}
  \item $\Phi$ is convex,

  \vrum

  \item $\Phi (0)=0$,

  \vrum

  \item $\Phi (t)<\infty$ for some $t>0$,

  \vrum

  \item $\lim
\limits_{t\rightarrow\infty} \Phi (t)=+\infty$.
\end{enumerate}
\end{defn}

\par

It is clear that $\Phi$ in
Definition \ref{Def:YoungFunc} is
non-decreasing, because if $0\leqslant t_1\leqslant t_2$
and $s\in [0,1]$ is chosen such
that $t_1=st_2$, then
\begin{equation*}
\Phi (t_1)=\Phi (st_2+(1-s)0)
\leqslant s\Phi (t_2)+(1-s)\Phi (0)
\leqslant \Phi (t_2),
\end{equation*}
since $\Phi (0) =0$ and $s\in [0,1]$.

\par

For any Young function, the
\emph{conjugate Young function}
is the function from
$[0,\infty)$ to $[0,\infty]$, given
by
\begin{equation}\label{Eq:YoungConj}
\Phi ^*(t) = \sup_{s>0} 
\left( s t - \Phi(s) \right).
\end{equation}
Notice that $\Phi ^*$ is also a Young
function.

\par

The Young functions $\Phi _1$ and $\Phi _2$
are called \emph{equivalent}, if there is
a constant $C\geqslant 1$ such that
\begin{alignat}{2}
C^{-1}\Phi _2(t)&\leqslant
\Phi _1(t)\leqslant C\Phi _2(t), &
\quad t&\in [0,\infty ).
\notag
\intertext{We recall that a Young function
is said to fulfill the \emph{$\Delta _2$-condition}
if there is a constant $C\geqslant 1$ such that}
\Phi (2t) &\leqslant C\Phi (t),&\qquad
t&\in [0,\infty ).
\label{Eq:Delta2Cond}
\intertext{Evidently, if $\Phi$
satisfies the $\Delta _2$-condition,
then $\Phi (t)<\infty$
when $t\geqslant 0$.
\linebreak
\indent
A Young
function is said to fulfill the
\emph{$\Lambda$-condition} if there is a $p>1$
such that}
\Phi(ct) &\leqslant c^p \Phi(t), &\qquad
t &\in [0,\infty ),\ c\in (0,1].  
\label{qphiCond}
\end{alignat}
Note that the $\Lambda$-condition
is also called the
\emph{lower Matuszewska-Orlicz index} in the 
literature
(cf. p. 117 in \cite{KaMaPe}).
The following characterization of Young
functions fulfilling the $\Delta _2$-condition
follows from the fact that any Young function is 
increasing. That is, $\Phi (t_1)\le \Phi (t_2)$
when $t_1\le t_2$. The verifications are
left for the reader.

\par

\begin{prop}
Let $\Phi$ be a Young function. Then
the following conditions are equivalent:
\begin{enumerate}
\item $\Phi$ satisfies the
$\Delta _2$-condition;

\vrum

\item for every constant $c>0$, the Young function
$t\mapsto \Phi (ct)$ is equivalent to $\Phi$;

\vrum

\item for some constant $c>0$ with $c\neq 1$,
the Young function
$t\mapsto \Phi (ct)$ is equivalent to $\Phi$.
\end{enumerate}
\end{prop}

\par

For any Young function $\Phi$, let $\Omega$
be the set
\begin{equation}\label{eq:OmegaDef}
\Omega =\Omega _\Phi \equiv
\sets {t>0}{0<\Phi (t)<\infty}
\end{equation}
and let the Lebesgue 
exponents $p_\Phi$ and $q_\Phi$ be given by
\begin{align}
p_{\Phi}
&\equiv
\begin{cases}
{\displaystyle{\sup_{t\in \Omega}
\left (\frac{t \Phi _\pm '(t)}{\Phi(t)}\right )}},
& \Omega=\mathbf{R}_+,
\\[1ex]
\infty, & \Omega\neq \mathbf{R}_+,   
\end{cases}
\label{Eq:StrictYoungFunc2}
\intertext{where $\mathbf{R}_+=(0,\infty)$, and}
q_{\Phi}
&\equiv
\begin{cases}
{\displaystyle{\inf_{t\in \Omega}
\left (\frac{t \Phi _\pm '(t)}{\Phi(t)}\right )}},
& \Omega\neq \emptyset,
\\
\infty, & \Omega = \emptyset.
\end{cases}
\label{Eq:StrictYoungFunc1}
\end{align}
These exponents
are essential in several investigations
of Orlicz spaces and are often called the
Simonenko indices (see e.{\,}g.
\cite{Mal1,Mal2,Liu}).
Here $\Phi _+'$ and $\Phi _-'$
are the right and left derivatives 
respectively.
Since $\Phi$ is convex, it follows
that these derivatives are
well-defined on $\Omega$.
We also put
$$
\Phi '_+(t_1)=\Phi '_-(t_2)=\infty 
\quad \text{when}\quad
t_2>t_1\ge t_0
\equiv
\sup \sets {t\ge 0}{\Phi (t)<\infty }. 
%
$$
Evidently, if the left limit $\Phi (t_0^-)$ 
of $\Phi$ at $t_0$ is finite,
then $\Phi '_-(t_0)$ is well-defined
and finite. If instead
$\Phi (t_0^-)=\infty$, then we let
$$
\Phi '_-(t_0)= \infty 
\quad
\left (
=\lim _{h\to 0^-}\frac 
{\Phi (t_0+h)-\Phi (t_0^-)}{h}
\right ).
$$


\begin{rem}
We have
\begin{equation}
\tag*{(\ref{Eq:StrictYoungFunc2})$'$}
p_\Phi = \begin{cases}
{\displaystyle{\sup_{t\in \Omega}
\left (\frac{t \Phi _\pm '(t)}{\Phi(t)}\right )}},
& |\Omega|=\infty,
\\[1ex]
\infty, & |\Omega|<\infty.  
\end{cases}
\end{equation}
\end{rem}
When $p_\Phi < \infty$, we observe that
for any $r_1,r_2>0$,
\begin{alignat}{3}
t^{p_\Phi}&\lesssim \Phi (t)
\lesssim t^{q_\Phi} &
\quad &\text{when} &\quad t&\leqslant r_1
\label{Eq:Squeezing1}
\intertext{and}
t^{q_\Phi}&\lesssim \Phi (t)\lesssim t^{p_\Phi} &
\quad &\text{when} &\quad
t&\geqslant r_2.
\label{Eq:Squeezing2}
\end{alignat}
In order to shed some light on this, as well as
demonstrate arguments used in the next section, we
here show these relations.

\par

By \eqref{Eq:StrictYoungFunc2} we obtain
$$
\frac {t\Phi '_+(t)}{\Phi (t)}-p_\Phi \leqslant 0
\quad \iff \quad
\left (
\frac {\Phi (t)}{t^{p_\Phi}}
\right ) '_{\!\! +}\leqslant 0.
$$
Hence $\Phi (t) = t^{p_\Phi}h(t)$ for some decreasing function $h(t)>0$.
This gives
$$
\Phi (t) = t^{p_\Phi}h(t)\geqslant t^{p_\Phi}h(r_1)\gtrsim t^{p_\Phi}
$$
for $t\leqslant r_1$ and 
$$
\Phi (t) = t^{p_\Phi}h(t)\leqslant t^{p_\Phi}h(r_2)\lesssim t^{p_\Phi}
$$
for $t\geqslant r_2$. This shows the relations between $t^{p_\Phi}$
and $\Phi (t)$ in \eqref{Eq:Squeezing1} and \eqref{Eq:Squeezing2}.
The remaining relations follow in similar ways.

\par

In our investigations we need to assume that
our Young functions are \emph{strict}
in the following sense.

\par

\begin{defn}\label{Def:StrictYoungFunc}
The Young function $\Phi$ from $[0,\infty)$ to $[0,\infty]$
is called \emph{strict} or a \emph{strict Young function},
if
\begin{enumerate}
  \item $\Phi (t)<\infty$ for every $t\in [0,\infty )$,

  \vrum

  \item $\Phi$ satisfies the $\Delta _2$-condition,

  \vrum

  \item $\Phi$ satisfies the $\Lambda$-condition. 
\end{enumerate}
\end{defn}

\par

In Section \ref{sec2} we give various kinds
of characterizations of the conditions (2)
and (3) in Definition \ref{Def:StrictYoungFunc}.
In particular we show that (2) and (3)
in Definition \ref{Def:StrictYoungFunc} are
equivalent to $p_\Phi <\infty$ and $q_\Phi >1$,
respectively.
(See Proposition \ref{Prop:qPhiEquivCond}.)

\par

It will also be useful to rely on regular Young functions, which
is possible due to the following proposition.

\par

\begin{prop}\label{Prop:SmoothEquivYoungFunc}
Let $\Phi$ be a Young function which satisfies the
$\Delta _2$-condition. For every $c\in (0,1)$,
there is a Young function
$\Psi$ such that the following is true:
\begin{enumerate}
\item $\Psi$ is equivalent to $\Phi$ and $\Psi \leqslant \Phi$;

\vrum

\item $\Psi$ is smooth on $\mathbf R_+$;

\vrum

\item $c\, \Phi _+'(0)\leqslant \Psi _+'(0)\leqslant \Phi _+'(0)$;
\vrum

\item $q_\Phi \le q_\Psi$.

\end{enumerate}
\end{prop}






\par

\begin{proof}
Let $\phi \in C_0^\infty [0,1]$ be such that
$\phi \geqslant 0$ and
$\int _0^1 \phi (s)\, ds=1$. Put
\begin{equation}\label{Eq:SmoothYoungFuncDef}
\Psi (t) = \int _0^1 \Phi (t-\textstyle{\frac 12}st)
\phi (s)\, ds.
\end{equation}
By straightforward computations it follows that $\Psi$ is a Young function. Let $C\ge1$ be as in
\eqref{Eq:Delta2Cond}. Since $\Phi$ is increasing
we obtain
$$
\frac 1C\Phi (t) \leqslant \Phi ({\textstyle{\frac 12}}t)
\leqslant
\int _0^1 \Phi (t-\textstyle{\frac 12}st)
\phi (s)\, ds
=
\Psi (t)
$$
and
$$
\Psi (t) = \int _0^1 \Phi (t-{\textstyle{\frac 12}}st)
\phi (s)\, ds
\leqslant
\int _0^1 \Phi (t)
\phi (s)\, ds = \Phi (t),
$$
and (1) follows.

\par

By \eqref{Eq:SmoothYoungFuncDef} we obtain
$$
\Psi (t) = \int _{t/2}^t \Phi (s)
\phi (2-2s/t)\frac 2t\, ds,
$$
which is evidently smooth on $\mathbf R_+$. This gives
(2).

\par

By differentiating \eqref{Eq:SmoothYoungFuncDef}
we obtain
$$
\Psi '_+(t) = \int _0^1
\Phi '_+(t-{\textstyle{\frac 12}}st)
(1-{\textstyle{\frac 12}}s)\phi (s)\, ds
\leqslant
\int _0^1 \Phi '_+(t)
\phi (s)\, ds = \Phi '_+(t).
$$
On the other hand, by choosing the support of $\phi$
to be sufficiently close to origin, we obtain
$$
\Psi '_+(0) = \int _0^1 \Phi '_+(0)
(1-\textstyle{\frac 12}s)\phi (s)\, ds
\geqslant
c\Phi '_+(0),
$$
and (3) follows.

\par

By \eqref{Eq:SmoothYoungFuncDef}
we also get
$$
t\Psi '_+(t) = \int _0^1
(t-{\textstyle{\frac 12}}st)
\Phi '_+
(t-{\textstyle{\frac 12}}st)
\phi (s)\, ds
\geqslant
q_\Phi \int _0^1
\Phi 
(t-{\textstyle{\frac 12}}st)
\phi (s)\, ds
=
q_\Phi \Psi (t),
$$
which implies that
$$
q_\Phi \leqslant \frac {t\Psi '_+(t)}{\Psi (t)}.
$$
The assertion (4) now follows
by taking the infimum of the
right-hand side.
\end{proof}


\par

It follows that $\Psi$ in Proposition
\ref{Prop:SmoothEquivYoungFunc} fulfills the
$\Delta _2$-condition, because $\Phi$ satisfies
that condition and $\Psi$ is equivalent to $\Phi$.

\par

\begin{defn}
Let $\Phi $ be a Young function.
The \emph{Orlicz space} $L^{\Phi }(\rr d)$ consists
of all Lebesgue measurable functions
$f:\rr d \rightarrow
\mathbf C$ such that
$$
    \Vert f\Vert_{L^{\Phi}}
    \equiv
    \inf 
    \Sets{\lambda>0}{\int_{\rr d} \Phi 
    \left (
    \frac{|f(x)|}{\lambda}
    \right )
    dx\leqslant 1}
$$
is finite.
\end{defn}

\par

\begin{defn}
Let $\Phi $ be a Young function.
The \emph{weak Orlicz space} $wL^{\Phi }(\rr d)$ consists
of all Lebesgue measurable functions
$f:\rr d \rightarrow
\mathbf C$ such that 
$$
\nm f{wL^{\Phi}}
\equiv
\inf \Sets { \lambda >0}
{\sup_{t>0} \left (
\Phi \left(\frac{t}{\lambda}\right)
\mu _f(t) \right ) \leqslant 1}
$$  
is finite. Here $\mu _f(t)$
is the Lebesgue measure of the set
$\sets {x\in \rr d}{|f(x)|>t}$.
\end{defn}

\par

As with the usual Lebesgue spaces,
$f,g\in wL^{\Phi}(\rr d)$ are equivalent
whenever $f=g$ a.{\,}e.

\par

\begin{rem}\label{Rem:EquivPhi}
Suppose that $p\in [1,\infty ]$. Then recall
from the introduction that
$L^{\Phi _{[p]}}(\rr d)=L^p(\rr d)$, when
$$
\Phi _{[p]}(t)= \frac {t^p}p, \ p<\infty 
\quad \text{and}\quad
\Phi _{[\infty ]}(t)
=
\begin{cases}
0, & 0\leqslant t \leqslant 1,
\\[1ex]
\infty, & t> 1.
\end{cases}
$$
Moreover, suppose
$p_1,p_2\in [1,\infty]$ and let
$$
\Phi (t)
=
\begin{cases}
{\frac {t^{p_2}}{p_2}}, & 0\leqslant t\leqslant 1,
\\[1ex]
{\frac {t^{p_1}}{p_1} +\frac 1{p_2}-\frac 1{p_1}},
& t> 1.
\end{cases}
$$
Here, we interpret
$\frac {t^\infty}{\infty}$
as
$$
\frac {t^\infty}{\infty}
\equiv
\lim _{p\to \infty}\frac {t^p}p
=
\begin{cases}
0, & 0\leqslant t \leqslant 1,
\\[1ex]
\infty , & t>1.
\end{cases}
$$
Then $\Phi$ is a Young function,
\begin{alignat*}{2}
L^\Phi (\rr d) &= L^{p_1}(\rr d)+L^{p_2}(\rr d),
&\quad p_1 &\leqslant p_2,
\intertext{and}
L^\Phi (\rr d)
&=
L^{p_1}(\rr d)\cap L^{p_2}(\rr d),&
\quad p_2 &\leqslant p_1.
\end{alignat*}

\par

In particular, for a general Young function
$\Phi$, it
follows from \eqref{Eq:Squeezing1}
and \eqref{Eq:Squeezing2} that
\begin{equation}\label{Eq:OrlLpInclusions}
L^{p_\Phi}(\rr d)
\cap
L^{q_\Phi}(\rr d)
\subseteq
L^\Phi (\rr d)
\subseteq
L^{p_\Phi}(\rr d)
+
L^{q_\Phi}(\rr d).
\end{equation}
\end{rem}

\par

\begin{rem}
Notice that the assignment $p_\Phi=\infty$ when $|\Omega|<\infty$ in \eqref{Eq:StrictYoungFunc2}$'$ is justified by \eqref{Eq:Squeezing2}. The assignment $q_\Phi=\infty$ when $\Omega=\emptyset$ in \eqref{Eq:StrictYoungFunc1} is justified by the observations in Remark \ref{Rem:EquivPhi} and the fact that for $\Phi(t)=\Phi_{[p]}(t)= \frac {t^p}{p}$, we have
$$
q_\Phi=\frac{t \Phi'(t)}{\Phi(t)} =p \to \infty, \quad \text{as } p \to \infty
$$ 
\end{rem}


\par

\subsection{Pseudo-differential operators}\label{subsec1.2}

\par

Let $\GL (d,\Omega )$ be the set of all
$d\times d$-matrices with
entries in the set $\Omega$, and let $a\in \mascS 
(\rr {2d})$ and $A\in \GL (d,\mathbf R)$ be fixed.
Then the pseudo-differential operator $\op _A(a)$
is the linear and
continuous operator on $\mascS  (\rr d)$, given by
\begin{equation}\label{e0.5}
(\op _A(a)f)(x)
=
(2\pi  ) ^{-d}\iint a(x-A(x-y),\xi )f(y)
e^{i\scal {x-y}\xi }\, dyd\xi ,
\end{equation}
when $f\in \mascS(\rr d)$. For
general $a\in \mascS '(\rr {2d})$, the
pseudo-differential operator $\op _A(a)$ is defined as the linear and
continuous operator from $\mascS (\rr d)$ to $\mascS '(\rr d)$ with
distribution kernel given by
\begin{equation}\label{atkernel}
K_{a,A}(x,y)=(2\pi )^{-\frac d2}
(\mascF _2^{-1}a)(x-A(x-y),x-y).
\end{equation}
Here $\mascF _2F$ is the partial Fourier transform of $F(x,y)\in
\mascS '(\rr {2d})$ with respect to the $y$ variable. We observe that $K_{a,A}$ in
\eqref{atkernel} makes sense as an
element in $\mascS '(\rr {2d})$
when $a\in \mascS '(\rr {2d})$,
since the mappings
\begin{equation}\label{homeoF2tmap}
\mascF _2\quad \text{and}\quad F(x,y)\mapsto F(x-A(x-y),x-y)
\end{equation}
are homeomorphisms on $\mascS '(\rr {2d})$.
In particular, the map $a\mapsto K_{a,A}$ is a homeomorphism on
$\mascS '(\rr {2d})$.

\par

An important special case appears when $A=t\cdot I$, with
$t\in \mathbf R$. Here and in what follows, $I\in \GL (d,\mathbf R)$ denotes
the $d\times d$ identity matrix. In this case we set
$$
\op _t(a) = \op _{t\cdot I}(a).
$$
The normal or Kohn-Nirenberg representation, $a(x,D)$, is obtained
when $t=0$, and the Weyl quantization, $\op ^w(a)$, is obtained
when $t=\frac 12$. That is,
$$
a(x,D) = \op _0(a)
\quad \text{and}\quad \op ^w(a) = \op _{1/2}(a).
$$

\par

For any $K\in \mascS '(\rr {d_1+d_2})$, we let $T_K$ be the
linear and continuous mapping from $\mascS (\rr {d_1})$
to $\mascS '(\rr {d_2})$, defined by the formula
\begin{equation}\label{pre(A.1)}
(T_Kf,g)_{L^2(\rr {d_2})} = (K,g\otimes \overline f )_{L^2(\rr {d_1+d_2})}.
\end{equation}
It is well-known that if $A\in \GL (d,\mathbf R)$, then it follows from Schwartz kernel
theorem that $K\mapsto T_K$ and $a\mapsto \op _A(a)$ are bijective
mappings from $\mascS '(\rr {2d})$
to the set of linear and continuous mappings from $\mascS (\rr d)$ to
$\mascS '(\rr d)$ (cf. e.{\,}g. \cite{Ho1}).

\par

In particular, for every $a_1\in \mascS  '(\rr {2d})$ and $A_1,A_2\in
\GL (d,\mathbf R)$, there is a unique $a_2\in \mascS  '(\rr {2d})$ such that
$\op _{A_1}(a_1) = \op _{A_2} (a_2)$. The following result explains the
relations between $a_1$ and $a_2$.

\par

\begin{prop}\label{Prop:CalculiTransfer}
Let $a_1,a_2\in \mascS '(\rr {2d})$ and $A_1,A_2\in \GL (d,\mathbf R)$.
Then
\begin{equation*} 
\op _{A_1}(a_1) = \op _{A_2}(a_2) \quad \iff \quad
e^{i\scal {A_2D_\xi}{D_x }}a_2(x,\xi )=e^{i\scal {A_1D_\xi}{D_x }}a_1(x,\xi ).
\end{equation*}
\end{prop}

\par

In \cite{Toft15}, a proof of the previous proposition is given, which is similar to
the proof of the case $A=t\cdot I$ in \cite{Ho1,Sh,Tr}.

\par

Let $r, \rho ,\delta \in \mathbf R$ be such that
$0\leqslant \delta \leqslant \rho \leqslant 1$ and $\delta <1$.
Then we recall that the H{\"o}rmander class
$S^r_{\rho ,\delta}(\rr {2d})$ consists of
all $a\in C^\infty (\rr {2d})$ such that
$$
\sum _{|\alpha |,|\beta |\leqslant N}\sup _{x,\xi \in \rr d}
\left (
\eabs \xi ^{-r+\rho |\alpha |-\delta |\beta |}|D_\xi ^\alpha D_x^\beta a(x,\xi )|
\right )
$$
is finite for every integer $N\geqslant 0$. Here we set
$\eabs \xi\equiv (1+|\xi|^2)^{\frac 12}$
when $\xi\in \rr d$.

\par

We recall the following continuity property 
for pseudo-differential
operators acting on $L^p$-spaces
(see e.{\,}g. \cite{Wong}).

\par

\begin{prop}\label{Prop:LpCont}
Let $p\in (1,\infty )$,
$A\in \GL (d,\mathbf R)$, and
$a\in S^0_{1,0}(\rr {2d})$. Then
$\op _A(a)$ is continuous on $L^p(\rr d)$.
\end{prop}

\par

In the next proposition we
essentially
recall H{\"o}rmander's
improvement of Mihlin's
Fourier multiplier theorem.

\par

\begin{prop}\label{Prop:HormMult}
Let $p\in (1,\infty )$ and
$a\in L^\infty (\rr d\setminus 0)$
be such that
\begin{equation}
\sup _{R>0}
\left (
R^{-d+2|\alpha |}\int _{A_R}
|\partial ^\alpha a(\xi )|^2\, d\xi
\right )
\end{equation}
is finite for every $\alpha \in \nn d$
with $|\alpha |\leqslant [\frac d2]+1$,
where $A_R$ is the annulus
$\sets {\xi \in \rr d}{R<|\xi |<2R}$.
Then $a(D)$ is continuous on $L^p(\rr d)$.
\end{prop}

\par

\subsection{Fourier integral operators of \texorpdfstring{$SG$}{SG} type} 
We recall that the so-called $SG$-symbol class $S^{m,\mu}(\rr{2d})$, $m,\mu\in\rr{}$,
consists of all $a\in C^\infty (\rr {2d})$ such that
$$
\sum _{|\alpha |,|\beta |\leqslant N}\sup _{x,\xi \in \rr d}
\left (
\eabs x ^{-m+|\alpha|} \eabs \xi ^{-\mu + |\beta |}|D_x ^\alpha D_\xi^\beta a(x,\xi )|
\right )
$$
is finite for every integer $N\geqslant 0$. Following \cite{CR2}, we say that
$\varphi\in C^\infty(\rr d\times(\rr d\setminus 0))$ is a phase-function if it is real-valued, positively $1$-homogeneous with respect to $\xi$, that is,
$\varphi(x,\tau\xi)=\tau\varphi(x,\xi)$ for all $\tau>0$, $x,\xi\in\rr d$, $\xi\not=0$, and satisfies, for all $x,\xi\in\rr d$, $\xi\not=0$,
\begin{equation}
    \label{eq:phf}
	\begin{alignedat}{2}
    |\det \varphi ''_{x,\xi}(x,\xi)|
    &\geqslant C>0,
    &\qquad
    \partial^\alpha_x\varphi(x,\xi)
    &\lesssim
    \eabs{x}^{1-|\alpha|} |\xi|,
    \quad
    \alpha\in \nn d,
    \\[1ex] 
    \eabs{\varphi^\prime_\xi(x,\xi)}
    &\asymp
    \eabs{x},&
    \qquad
    \eabs{\varphi^\prime_x(x,\xi)}
    &\asymp
    \eabs{\xi}.
	\end{alignedat}
\end{equation}
In the sequel, we will denote the set of all such phase-functions by $\phf$.

\par

For any $a\in S^{m,\mu}(\rr{2d})$ and $\varphi\in\phf$,
the Fourier integral operator $\Op_{\varphi}(a)$ is
the linear and continuous operator from $\mascS (\rr d)$
to $\mascS '(\rr d)$, given by
\begin{equation}\label{eq:fio2}
        (\Op_\varphi(a)f)(x)=\int_{\rr d} 
        e^{i\varphi(x,\xi)} a(x,\xi) \widehat{f}(\xi)\, d\xi,
        \quad f\in\mathscr{S}(\rr d).
\end{equation}
We recall the following (global on $\rr d$) $L^p$-boundedness result, 
proved in \cite{CR2}.   

\begin{thm}\label{Thm:LpSGFIOcont}
Let $p\in (1,\infty)$ and $m_0,\mu _0
\in \mathbf R$ be such that
    \begin{equation}\label{eq:sharp-thresholds}
     m_0,\mu_0
     \leqslant -(d-1)\left|\frac{1}{p}-\frac12\right|.
    \end{equation}
%
Suppose that $\varphi\in\phf$ and $a\in S^{m_0,\mu_0}(\rr {2d})$
is such that $|\xi|\geqslant\varepsilon$, for some $\varepsilon>0$, 
on the support of $a$. Then  $\Op_{\varphi}(a)$ 
from  $\mascS (\rr d)$ to $\mascS '(\rr d)$ extends uniquely
to a continuous operator from $L^p(\rr d)$ to itself.
\end{thm}

\par

\begin{rem}
If $\varphi(x,\xi)=\scal x\xi$
in Theorem \ref{Thm:LpSGFIOcont},
then $\op _\fy (a)$ becomes a
pseudo-differential operator with symbol 
$a$, which we still denote by $\Op(a)$. 
In this case, Proposition
\ref{Prop:LpCont} is a strict improvement
of Theorem \ref{Thm:LpSGFIOcont}.
For $p\neq 2$ and general $\varphi \in \phf$, 
uniform boundedness of the amplitude $a$
is not enough to guarantee that 
$\Op_\varphi(a)$ maps $L^p$ continuously
into itself, even if the support of
$f$ is compact.
(See \cite{SSS91}.)
In \eqref{eq:sharp-thresholds}, the condition on the $x$-order $m_0$ can be viewed as a 
\emph{loss of decay}. Similarly, the condition on the $\xi$-order $\mu_0$ can be considered a \emph{loss of 
smoothness}. Notice also that no
condition of compactness of the support of $f$ is 
needed in Theorem \ref{Thm:LpSGFIOcont} (see 
\cite{CR2} and the references therein for 
more details). Furthermore, we remark that in \cite{DSFS} 
many local and global $L^p$-boundedness results for 
Fourier integral operators have been proved under hypotheses 
different from those assumed in Theorem \ref{Thm:LpSGFIOcont}.
\end{rem}


\section{The role of upper and lower
Lebesgue exponents
for Young functions}\label{sec2}

\par


In this section we investigate
the Orlicz Lebesgue exponents $p_\Phi$
and $q_\Phi$ and link conditions on
these exponents to various properties
of their Young functions $\Phi$. Especially
we show that both implications in
\eqref{eq:qPhi} involving $q_\Phi$
are wrong (see Proposition
\ref{Prop:YounFuncEquivCond2}).
Instead we
deduce other conditions on $\Phi$ which
characterize $q_\Phi >1$ (see Propositions
\ref{Prop:YounFuncEquivCond} and
\ref{Prop:qPhiEquivCond}). We also
remark that our investigations are
related to the achievements in
\cite{KovSkr}.

\par

In the following proposition we list some basic
properties of relations between Young functions
and their upper and lower Lebesgue exponents.

\par

\begin{prop}\label{Prop:YounFuncEquivCond}
Let $\Phi$ be a Young function with $\Omega$ as in \eqref{eq:OmegaDef},
and let
$p_{\Phi}$ and $q_{\Phi}$ be as in
\eqref{Eq:StrictYoungFunc2} and
\eqref{Eq:StrictYoungFunc1}. Then the following is 
true:
\begin{enumerate}
\item $1\leqslant q_\Phi \leqslant p_\Phi$;

\vrum

\item $p_\Phi =1$, if and only if $\Phi$ is a linear map;

\vrum

\item $p_{\Phi} < \infty$, if and only if $\Phi$ 
fulfills
the $\Delta_2$-condition;

\vrum

\item $q_{\Phi} > 1$, if and only if there is a
$p>1$ such that $\frac{\Phi(t)}{t^p}$ increases.
\end{enumerate}
\end{prop}

\par

\begin{rem}\label{Rem:YounFuncEquivCond}
Taking into account that $\Phi$ in Proposition
\ref{Prop:YounFuncEquivCond} is a Young function,
we find that (4) is equivalent to 
\begin{enumerate}
\item[{\rm{(4)$'$}}]
\emph{$q_{\Phi} > 1$, if and only if there is a
$p>1$ such that $\frac{\Phi(t)}{t^p}$ increases,
$$
\lim _{t\to a^+} \frac{\Phi(t)}{t^p} = 0
\quad \text{and}\quad
\lim _{t\to \infty} \frac{\Phi(t)}{t^p} = \infty
$$
for some $a\geqslant 0$.
}
\end{enumerate}
\end{rem}

\par

Most of Proposition \ref{Prop:YounFuncEquivCond}
and Remark \ref{Rem:YounFuncEquivCond} are
well-known (see e.{\,}g. \cite{BlaUst,Mal1,Mal2,Mal3}).
In order to be self-contained, we here present a proof.

\par

\begin{proof}[Proof of Proposition
\ref{Prop:YounFuncEquivCond}]
If 
$\Omega = \emptyset$, (1) is trivially true. Assume 
therefore that $\Omega \neq \emptyset$, and let
$t\in \Omega$. By the fact that $\Phi$
is convex on $\Omega$, we obtain
$$
\frac {\Phi (t)}t =\frac {\Phi (t)-\Phi (0)}t
\leqslant \Phi '_-(t)\leqslant \Phi' _+(t).
$$
This gives (1).

\par

If $\Phi$ is linear, then $\frac {t\Phi '(t)}{\Phi (t)}=1$ for all $t\in\mathbf{R}_+$, giving that
$q_\Phi =p_\Phi =1$. Suppose instead that $p_\Phi =1$. Then $\Omega = \mathbf{R}_+$ and
$$
\frac {t\Phi_\pm '(t)}{\Phi (t)}=1,
$$ 
for all $t\in\mathbf{R}_+$ in view of (1) and its proof. This implies that $\Phi (t)=Ct$ for some
constant $C$, and (2) follows.

\par

In order to prove (3), we first suppose that $p_{\Phi} < \infty$. Then
$$
\frac{t \Phi _\pm'(t)}{\Phi(t)}\leqslant R
\quad \iff \quad
t \Phi _\pm'(t)-R\Phi(t) \leqslant 0,
$$
for some $R>0$. Since $\Phi (0)=0$, we obtain
$$
\Phi (t) = t^Rh(t),\quad t>0,
$$
for some positive decreasing function $h(t)$, $t>0$. This gives
$$
\Phi (2t) = (2t)^Rh(2t) \leqslant 2^R t^Rh(t)=2^R\Phi (t),
$$
and it follows that $\Phi$ satisfies the $\Delta _2$-condition
when $p_{\Phi} < \infty$.

\par

Suppose instead that $\Phi$ satisfies the $\Delta _2$-condition. Then $\Omega = \mathbf{R}_+$.
By the mean-value theorem and the fact that $\Phi '_\pm(t)$
is increasing we obtain
$$
\Phi '_\pm(t)t \leqslant \Phi (2t)-\Phi (t)\leqslant \Phi (2t)\leqslant C\Phi (t),
$$
for some constant $C>0$. Here the last inequality follows from
the fact that $\Phi$ satisfies the $\Delta _2$-condition. This gives
$$
\frac{t \Phi _\pm'(t)}{\Phi(t)} \leqslant C,
$$
giving that $p_\Phi \leqslant C<\infty$, and we have proved (3).

\par

Next we prove (4). Suppose that $q_\Phi>1$. Then
there is a $p>1$ such that
$$
\frac{t \Phi_\pm'(t)}{\Phi(t)} > p
$$
for all $t\in\Omega$, which gives
$$
t \Phi_\pm'(t) - p \Phi(t) > 0.
$$
Hence
$$
\frac{t^p \Phi_\pm'(t) - p t^{p-1}
\Phi(t)}{t^{2p}}> 0,
$$
or equivalently
$$
\left(\frac{\Phi(t)}{t^p} 
\right)_{\!\!\pm}' > 0.
$$
Hence, the result now holds. If we instead suppose
that $\frac{\Phi(t)}{t^p}$ is increasing for some
$p>1$, then applying the arguments above in reverse
order now yields $q_\Phi\geqslant p > 1$.
\end{proof}

\par

We observe that, besides
Proposition 
\ref{Prop:YounFuncEquivCond} (3),
there are several
contributions on characterizations
of the $\Delta _2$-condition
(see e.{\,}g. 
\cite{BlaUst,Mal1,Mal2,Mal3,KovSkr}).

\par

For the equivalence in (4) of Proposition
\ref{Prop:YounFuncEquivCond} we note further.

\par

\begin{prop}\label{Prop:qPhiEquivCond}
Let $\Phi$ be a Young function with $\Omega$ as in \eqref{eq:OmegaDef},
and let $q_{\Phi}$ be as in
\eqref{Eq:StrictYoungFunc1}.
Then the following
conditions are equivalent:
\begin{enumerate}
\item $q_\Phi >1$;

\vrum

\item there is a $p>1$ such that
$\frac{\Phi(t)}{t^p}$ increases;

\vrum

\item there are $p,q>1$ such that
$\frac{\Phi(t)}{t^{p}}$ increases
near the origin and
$\frac{\Phi(t)}{t^{q}}$ increases
at infinity;

\vrum

\item $\Phi$ fulfills the $\Lambda$-condition.
\end{enumerate}
\end{prop}

\par

\begin{proof}
The result is obviously true when
$\Omega =\emptyset$. 
Therefore, suppose
that $\Omega \neq \emptyset$.
The equivalence of (1) and (2) was established in 
Proposition \ref{Prop:YounFuncEquivCond}. Trivially, 
(2) implies (3). Moreover, $\frac{\Phi(t)}{t^p}$ 
increases if and only if for any $t>0$ and any $c\in 
(0,1]$,
$$
\frac{\Phi(ct)}{(ct)^p} \leqslant \frac{\Phi(t)}{t^p}
$$
which is equivalent to (4), hence (2) is 
equivalent 
to (4).
We now show that (3)
implies (1), 
yielding the result.

\par

\par

Suppose that (3) holds. First
we also suppose that
$|\Omega|=\infty$.
Then, there are $R_1,R_2>0$ 
such that $\frac {\Phi(t)}{t^p}$
is increasing in
$\Omega _1\equiv (0,R_1)\cap \Omega$,
and that $\frac {\Phi(t)}{t^p}$ is increasing
in
$\Omega _2\equiv  (R_2,\infty)\cap \Omega$.
By test of differentiation, we obtain
\begin{equation}\label{Eq:q1q2Def}
q_1 = \inf_{t\in \Omega _1}
\left (
\frac{t \Phi_\pm'(t)}{\Phi(t)}
\right )
\geqslant p>1\quad\text{and}\quad q_2 =
\inf_{t\in \Omega _2}
\left (
\frac{t \Phi_\pm'(t)}{\Phi(t)}
\right )
\geqslant q>1.
\end{equation}
Let
\begin{equation}\label{Eq:q12Def}
q_{1,2} = \inf _{t\in \Omega _{1,2}}
\left (
\frac{t \Phi_\pm'(t)}{\Phi(t)}
\right ),
\quad \text{where}\quad
\Omega _{1,2}=[R_1,R_2]
\cap \Omega .
\end{equation}

\par

We intend to show that $q_2>1$, 
which in turn 
yield
$q_\Phi
= \min \{ q_1,q_{1,2},q_2\} > 1$, 
completing the proof in this case. Here we put $\inf \emptyset =\infty$. 

\par

Let $\varphi_1(t) = k_1 t - m_1$ and
$\varphi_2(t) = k_2 t - m_2$, with
$k_j = \Phi_+'(R_j)$ and $m_j$ chosen
so that $\varphi_j(R_j) = \Phi(R_j)$,
$j=1,2$. Given that $\Phi$ is a Young
function, is convex, and fulfills (3),
it is clear that $k_1 \leqslant k_2$,
$m_1 \leqslant m_2$,
and $m_j > 0$ for $j=1,2$.

\par

We now approximate $\Phi(t)$ with linear segments forming 
polygonal chains for $R_1 \leqslant t \leqslant R_2$. Pick 
points $R_1 = t_0 < t_1 < \dots < t_n = R_2$ and define 
functions $f_j(t) =a_j t - b_j $ such that $f_j(t_j) = 
\Phi(t_j)$ and $f_j(t_{j+1})=\Phi(t_{j+1})$. Let $\Phi_n(t)$ 
be the polygonal chain on $[R_1,R_2]$ formed by connecting 
the functions $f_j$, meaning $\Phi_n(t) = f_j(t)$ whenever 
$t\in [t_j,t_{j+1}]$.

\par

Since $\Phi$ is convex and increasing, we have $k_1 
\leqslant a_j \leqslant k_2$ and $m_1 \leqslant b_j 
\leqslant m_2$ for all $j=1,\dots, n$.
Hence, for any $j = 1, \dots, n$, 
$$
\inf_{t\in[t_j,t_{j+1}]} \left( 
\frac{t (f_j)_\pm'(t)}{f_j(t)} \right)
= 
\inf_{t\in[t_j,t_{j+1}]}
\left(
1 + \frac{b_j}{a_j t - b_j} 
\right)
\geqslant 1 + \frac{m_1}{\Phi(R_2)},
$$
where the last inequality follows from the fact that $b_j \geqslant 
m_1$ and $a_j t_j - b_j = f_j(t_j) \leqslant f_n(t_n) = 
\Phi(R_2)$. From this, it is clear that
$$
q_{\Phi_n} = \inf_{t\in[R_1,R_2]}
\left (
\frac{t (\Phi_n)_\pm'(t)}{\Phi_n(t)} 
\right )
\geqslant 1 + \frac{m_1}{\Phi(R_2)}
$$
independent of the choice of $n$ and the points $t_j$,
$j=1,\dots n-1$, and therefore
$$
q_2 = \lim_{n\rightarrow \infty} q_{\Phi_n} \geqslant 1 + \frac{m_1}{\Phi(R_2)} > 1.
$$
This gives (1), completing the proof when $|\Omega |=\infty$.

\par

Next we consider the case
when $|\Omega |<\infty$.
Then $a=\sup \Omega$
is finite because $\Phi$ is increasing.
First suppose that
$$
\lim _{t\to a^-}\Phi (t) <\infty .
$$
Letting $R_2=a$, $q_1$, and
$q_{1,2}$ be as in \eqref{Eq:q1q2Def}
and \eqref{Eq:q12Def},
the same  arguments as in the previous
case show that $q_{1,2}>1$. This
gives
$q_\Phi = \min \{ q_1,q_{1,2}\} >1$,
and (1) follows, giving the
result in this case as well.

\par

It remains to prove the result
when $|\Omega |<\infty$ and
$$
\lim _{t\to a^-}\Phi (t) =\infty .
$$
We may assume that $R_1<a$, and
we let $a_\ep =a-\ep$ when
$\ep \in (0,a-R_1)$. By convexity
we have
$$
\frac {\Phi (a_\ep )-\Phi (R_1)}
{a_\ep -R_1}
\leqslant
\Phi '_-(a_\ep),
$$
which gives
$$
\frac {a_\ep \Phi '_-(a_\ep)}
{\Phi (a_\ep )}
\geqslant
\frac {a_\ep \left (
1-\frac {\Phi (R_1)}{\Phi (a_\ep )}
\right )}
{a_\ep -R_1}
\to \frac a{a-R_1}>1,
$$
as $\ep \to 0+$. Hence, if
$c\in (1,\frac a{a-R_1} )$, then
there is a $\delta \in (0,a-R_1)$
such that
\begin{equation}\label{Eq:Newq2Est}
\frac {t\Phi _-'(t)}{\Phi (t)}>c>1,
\quad \text{when}\quad
t\in (a-\delta ,a).
\end{equation}

\par

Let $R_2=a-\delta$, $\Omega _2$
be redefined as
$\Omega _2=(a-\delta ,a)$,
and let $q_1$, $q_2$, and
$q_{1,2}$ be as in
\eqref{Eq:q1q2Def} and
\eqref{Eq:q12Def}. Then by
the same arguments as in the
first case in combination
with \eqref{Eq:Newq2Est},
we obtain
$$
q_1
\geqslant
p>1,
\quad
q_{1,2}>1,
\quad \text{and}\quad
q_2>1.
$$
This gives
$q_\Phi =\min
\{ q_1,q_{1,2},q_2 \}>1$,
leading to (1) in this case as well,
and the result follows.
\end{proof}

\par

\begin{rem}
Due to Proposition \ref{Prop:qPhiEquivCond}
and its proof it is evident that $p$ in 
Proposition \ref{Prop:qPhiEquivCond} (4) is
strongly linked to the lower
Matuszewska-Orlicz index, given in
e.{\,}g. p. 117 in \cite{KaMaPe}. It
follows that parts of
Propositions \ref{Prop:YounFuncEquivCond}
and \ref{Prop:qPhiEquivCond} follow from
some of the established
properties on p. 118--121 in \cite{KaMaPe},
at least after suitable computations.
\end{rem}

\par

The following proposition shows that the condition
$q_\Phi  >1$ cannot be linked to strict convexity for
the Young function $\Phi$.

\par

\begin{prop}\label{Prop:YounFuncEquivCond2}
Let $\Phi$ be a Young function,
which is non-zero
outside the origin, and let
$q_{\Phi}$ be as in
\eqref{Eq:StrictYoungFunc1}.
Then the following is true:
\begin{enumerate}
\item if $q_{\Phi} > 1$, then there is an
equivalent Young function to $\Phi$
which is strictly  convex;

\vrum

\item $\Phi$ can be chosen such that
$q_\Phi >1$ but $\Phi$ is not strictly convex;

\vrum

\item $\Phi$ can be chosen such that
$q_\Phi =1$ but $\Phi$ is strictly convex.


\end{enumerate}
\end{prop}

\par


\par

\begin{rem}
In \cite{Liu} it is stated that (1) in
Proposition \ref{Prop:YounFuncEquivCond2}
can be replaced by
\begin{enumerate}
\item[{\rm{(1)$'$}}]$ q_\Phi >1$, if and only if
$\Phi$ is strictly convex.
\end{enumerate}
This is equivalent to that the following conditions
should hold:
\begin{enumerate}
\item[{\rm{(2)$'$}}] if $q_\Phi >1$,
then $\Phi$ is strictly convex;

\vrum

\item[{\rm{(3)$'$}}] if $\Phi$ is
strictly convex, then $q_\Phi >1$.
\end{enumerate}
(See remark after (1.1) in \cite{Liu}.)
Evidently, the assertion in \cite{Liu} is 
(strictly) stronger than
Proposition \ref{Prop:YounFuncEquivCond2} (1).
On the other hand,
Proposition \ref{Prop:YounFuncEquivCond2} (2)
shows that (2)$'$ can not be true and
Proposition \ref{Prop:YounFuncEquivCond2} (3)
shows that (3)$'$ can not be true.
Consequently, both implications in (1)$'$
are false.
\end{rem}

\par

\begin{proof}[Proof of Proposition
\ref{Prop:YounFuncEquivCond2}]
We begin by proving (1). 
Therefore assume that
$q_\Phi>1$. By Proposition \ref{Prop:SmoothEquivYoungFunc}, we can assume that $\Phi$ is smooth on $\mathbf{R} _+$.
Suppose that $\Phi$ fails to be
strictly convex in the whole interval
$(0,\ep )$, for some
$\ep >0$. 
This implies that $\Phi (t)=ct$ when $t\in (0,\ep )$,
for some $c\geqslant 0$, which in turn gives
$q_\Phi=1$, violating the condition $q_\Phi>1$.
Hence $\Phi$ must be strictly convex in $(0,\ep )$,
for some choice of $\ep >0$. 

\par

Let
$$
\Psi (t) = \int _0^t\Phi (t-s)e^{-s}\, ds.
$$
Then
$$
\Psi ''(t) = \Phi '_+(0)e^{-t}+\int _0^t\Phi ''(t-s)e^{-s}
\, ds
\geqslant
\int _{t-\ep}^t\Phi ''(t-s)e^{-s}\, ds >0,
$$
since $\Phi ''(t-s)>0$ and is continuous when $s\in (t-\ep ,t)$.
This shows that $\Psi$ is a strictly convex Young
function.

\par

Since $\Phi$ is increasing we also have
$$
\Psi (t)\leqslant \Phi (t),
$$
because
$$
\Psi (t) = \int _0^t\Phi (t-s)e^{-s}\, ds
\leqslant
\Phi (t)\int _0^te^{-s}\, ds
\leqslant
\Phi (t)\int _0^\infty e^{-s}\, ds
=
\Phi (t).
$$
This implies that
$$
\Phi _1(t)\equiv \Phi (t)+\Psi (t)
$$
is a Young function equivalent to $\Phi (t)$. Since 
$\Psi$ is strictly convex, it follows that
$\Phi _1$ is strictly convex as well.
Consequently, $\Phi _1$ fulfills the
required conditions for the searched Young
function, and (1) follows.

\par
In order to prove (2), we choose
$$
\Phi (t)=
\begin{cases}
  2t^2,&\text{when}\ t \leqslant 1
  \\[1ex]
   4t-2 ,&\text{when}\ 1 \leqslant t \leqslant 2
    \\[1ex]
   t^2+2,&
   \text{when}\ 
   t \geqslant 2 
\end{cases}
$$
which is not strictly convex. Then
\begin{align*}
q_{\Phi}
&=
\inf _{t>0}
\left (
\frac{t \Phi'(t)}{\Phi(t)}
\right )
\\[1ex]
&=
\min \left \{
\inf_{t \leqslant 1} \left (\frac{4t^2}{2t^2}\right ),
\inf_{1 \leqslant t \leqslant 2} \left (\frac{4t}{4t-2}\right ),
\inf_{t \geqslant 2} \left (\frac{2t^2}{t^2+2}\right )  
\right \}
=
\frac{4}{3} >1,
\end{align*}
which shows that $\Phi$ satisfies all the searched properties.
This gives (2).

\par

Next we prove (3). Let
$$
\Phi (t) =
t\ln (1+t),\quad  t\geqslant 0.
$$
Then $\Phi$ is a Young function,
and it follows by straight-forward computations that
$q_{\Phi} = 1$.
We also have $\Phi ''(t) >0$,
giving that $\Phi$ is strictly convex. 
Consequently, $\Phi$ satisfies all searched 
properties and (3) follows.
This gives the result.
\end{proof}

\par

The exponents $q_\Phi$ and $p_\Phi$
can also be related to corresponding
exponents for the conjugate Young
function as in the following proposition,
which slightly generalizes (2.11) in
\cite{KovSkr}.

\par

\par

\begin{thm}\label{thm:conjYoungImproved}
Suppose $\Phi$ is a Young function and $\Psi$ its corresponding
conjugate Young function.
Then
$$
\frac{1}{p_\Psi} + \frac{1}{q_\Phi} = 1.
$$
\end{thm}

\par

\begin{proof}
Suppose that for every $t=t_0\in(0,\infty)$ there is some $s_0 =s(t)$ such that 
$$
\Psi(t) = s t - \Phi(s).
$$
Observe that, necessarily, $\Psi(t)<\infty$.
Since for any $h\in \mathbf{R}$,

$$
\Psi(t_0+h) = \sup_{s>0} 
\left( s(t_0+h) - \Phi(s) \right) \geqslant s_0(t_0+h) - \Phi(s_0),
$$
we obtain
$$
\frac{\Psi(t_0 + h) - \Psi(t_0)}{h} \leqslant s_0
$$
when $h<0$ and
$$
\frac{\Psi(t_0 + h) - \Psi(t_0)}{h} \geqslant s_0
$$
when $h>0$.
By taking limits, we therefore have $\Psi_-'(t_0) \leqslant s_0 \leqslant \Psi_+'(t_0)$.
Hence
$$
\frac{t_0 \Psi_-'(t_0)}{\Psi(t_0)} 
\leqslant \frac{t_0 s_0}{t_0 s_0 - \Phi(s_0)}
\leqslant \frac{t_0 \Psi_+'(t_0)}{\Psi(t_0)}
$$
and therefore
$$
\frac{1}{
\left(\frac{t_0 \Psi_+'(t_0)}{\Psi(t_0)}
\right)} 
\leqslant 1 -\frac{\Phi(s_0)}{t_0 s_0}
\leqslant \frac{1}{
\left(\frac{t_0 \Psi_-'(t_0)}{\Psi(t_0)}
\right)}. 
$$
By rewriting $\Psi(t_0) = t_0 s_0 - \Phi(s_0)$ as $\Phi(s_0) = s_0 t_0 - \Psi(t_0)$ and performing the exact same arguments as above, we obtain (for the same $s_0,t_0$) that $\Phi_-'(s_0) \leqslant t_0 \leqslant \Phi_+'(s_0)$. Hence, for every $t_0$ and corresponding $s_0$, we arrive at the inequalities
\begin{equation}\label{eq:conjIneq1}    
\frac{1}{
\left(\frac{t_0 \Psi_+'(t_0)}{\Psi(t_0)}
\right)} 
\leqslant 1 -\frac{1}{\left(\frac{s_0\Phi_+'(s_0)}{\Phi(s_0)}\right)}
\end{equation}

and
\begin{equation}\label{eq:conjIneq2}
1 -\frac{1}{\left(\frac{s_0\Phi_-'(s_0)}{\Phi(s_0)}\right)}
\leqslant \frac{1}{
\left(\frac{t_0 \Psi_-'(t_0)}{\Psi(t_0)}
\right)}.
\end{equation}
By maximizing the right-hand side of \eqref{eq:conjIneq1} and then maximizing the left-hand side, it is clear that
$$
\frac{1}{q_\Psi} \leqslant 1 - \frac{1}{p_\Phi}.
$$
But the same procedure applied to \eqref{eq:conjIneq2} yields
$$
1 - \frac{1}{p_\Phi} \leqslant \frac{1}{q_\Psi}.
$$
This proves the desired equality in the case that, for every $t\in\mathbf{R}_+$, the supremum of $(s t - \Phi(s))$ is attained for some $s\in\mathbf{R}_+$.

\par

Now suppose instead that there is some $t=t_0\in\mathbf{R}_+$ such that
$$
\Psi(t_0) = \sup_{s>0} \left(s t_0 - \Phi(s)\right) = \lim_{s\to\infty} \left(s t_0 - \Phi(s)\right).
$$
Then
$$
\Psi(t_0+\varepsilon)\geqslant \lim_{s\to\infty} \left(s t_0 - \Phi(s) + \varepsilon s\right) = \infty,
$$ 
meaning $p_\Psi = \infty$.
Moreover,
$$
s t_0 - \Phi(s) \rightarrow \Psi(t_0), \quad s\to\infty
$$
implies that
$$
\frac{\Phi(s)}{s} \rightarrow t_0, \quad s\to\infty
$$
meaning
$$
\frac{\Phi(s)}{s^p} \rightarrow 0, \quad s\to\infty
$$
for all $p>1$. Hence
$q_\Phi = 1$ by Proposition \ref{Prop:YounFuncEquivCond}(4). This completes the proof.
\end{proof}

\par

\begin{cor}\label{Cor:EquivYoungFuncLebesgueExp}
Let $\Phi_1$ and $\Phi_2$ be equivalent Young functions. Then the following is true:
\begin{enumerate}
\item $p_{\Phi_1} < \infty$ if and only if $p_{\Phi_2}<\infty$;

\vrum

\item $q_{\Phi_1} > 1$ if and only if $q_{\Phi_2}>1$.
\end{enumerate}
\end{cor}

\par

\begin{proof}
If $p_{\Phi_1} < \infty$, then by Proposition \ref{Prop:YounFuncEquivCond} (3), $\Phi_1$ fulfills the $\Delta_2$-condition. 
Since $\Phi_2$ is equivalent to $\Phi_1$, $\Phi_2$ also fulfills the $\Delta_2$-condition. 
This proves (1).

\par

Suppose $q_{\Phi_1} > 1$. 
Then, by Theorem \ref{thm:conjYoungImproved}, 
$p_{\Psi_1} < \infty$, where 
$\Psi_1$ is the conjugate Young function to $\Phi_1$. Since $\Psi_1$ is equivalent to $\Psi_2$, 
where $\Psi_2$ is the conjugate Young function to $\Phi_2$, we get $p_{\Psi_2} < \infty$ by (1). 
Applying Theorem \ref{thm:conjYoungImproved} once more, we obtain $q_{\Psi_2} > 1$. This completes the proof.
\end{proof}

\par








\par

\section{Continuity for pseudo-differential
operators, Fourier multipliers, and Fourier 
integral operators on
Orlicz spaces and applications to PDEs}\label{sec3}

\par

In this section we extend
properties on $L^p$ continuity
for various types of
Fourier type operators into
continuity on Orlicz spaces.
Especially we perform such extensions
for H{\"o}rmander's improvement
of Mihlin's Fourier multiplier theorem
(see Theorem \ref{Thm:OrlContHormMult}).
We also deduce Orlicz space continuity
for suitable classes of pseudo-differential and Fourier integral
operators (see Theorems \ref{Thm:OrlContPseudo} and \ref{Thm:OrlContSGFIOs}).
Our investigations are
based on a special case of Marcinkiewicz
type interpolation theorem for Orlicz
spaces, deduced in \cite{Liu}.
Subsequently, we introduce a scale of Sobolev spaces modelled on Orlicz spaces,
and show boundedness results involving them, relying on our $L^\Phi$ continuity 
results and on lift properties. Finally, we illustrate a notion of (global) wave-front set
associated with Sobolev-Orlicz spaces, along the lines in \cite{CJT13b} (see also \cite{CJT13a}),
and show its propagation for suitable classes of PDEs.

\par


\par

\subsection{Continuity results} We now recall
the following Marcinkiewicz type interpolation
theorem on Orlicz spaces. We omit the proof since
the result is a special case of
\cite[Theorem 5.1]{Liu}.

\par

\begin{prop}\label{Prop:OrliczIntPol}
Let $\Phi$ be a strict Young function and
$p_0,p_1 \in (0,\infty ]$
be such that
$p_0<q_{\Phi} \leqslant p_{\Phi}<p_1$, where
$p_{\Phi}$ and $q_{\Phi}$ are defined in
\eqref{Eq:StrictYoungFunc2} and
\eqref{Eq:StrictYoungFunc1}.
Also let
\begin{equation}\label{Eq:OrliczIntPol}
T : L^{p_0}(\rr d) + L^{p_1}(\rr d)
\to
wL^{p_0}(\rr d) + wL^{p_1}(\rr d)
\end{equation}
be a linear and continuous map
which restricts to linear and continuous
mappings
\begin{alignat}{5}
T &:&\,  L^{p_0}(\rr d)
&\to
\, wL^{p_0}(\rr d) &
\qquad &\text{and} &\qquad
T &:& L^{p_1}(\rr d)
&\to
wL^{p_1}(\rr d).
\notag
\intertext{Then \eqref{Eq:OrliczIntPol} restricts
to linear and continuous mappings}
T &:& L^{\Phi}(\rr d)
&\to
L^{\Phi}(\rr d)&
\qquad &\text{and} &\qquad
T &:& \ wL^{\Phi}(\rr d)
&\to
wL^{\Phi}(\rr d).
\label{Eq:InterpolCont}
\end{alignat}
\end{prop}

\par

\begin{rem}
Let $\Phi$ and $T$ be the same as in
Proposition \ref{Prop:OrliczIntPol}.
Then the continuity of the mappings
in \eqref{Eq:InterpolCont} means
\begin{alignat*}{2}
\nm {Tf}{L^{\Phi}}
&\lesssim
\nm f{L^{\Phi}}, &
\quad f &\in L^{\Phi}(\rr d)
\intertext{and}
\nm {Tf}{w L^{\Phi}}
&\lesssim
\nm f{w L^{\Phi}}, &
\quad f &\in w L^{\Phi}(\rr d).
\end{alignat*}
\end{rem}

\par


\par

A combination of Propositions
\ref{Prop:LpCont} and
\ref{Prop:OrliczIntPol}
gives the following result on continuity
properties for pseudo-differential operators
on $L^{\Phi}$-spaces.

\par

\begin{thm}\label{Thm:OrlContPseudo}
Let $\Phi$ be a strict Young function,
$A \in \GL (d,\mathbf R )$, and
$a\in S^0_{1,0}(\rr {2d})$. Then
$$
\op _A(a) : L^{\Phi}(\rr d) \to L^{\Phi}(\rr d)
\quad \text{and}\quad
\op _A(a) : wL^{\Phi}(\rr d) \to wL^{\Phi}(\rr d)
$$
are continuous.
\end{thm}

\par

\begin{proof}
By Propositions \ref{Prop:YounFuncEquivCond} and \ref{Prop:qPhiEquivCond} it follows
that $q_{\Phi}>1$ and $p_{\Phi}< \infty$.
Choose $p_0,p_1\in (1,\infty )$
such that $p_0 <q_{\Phi}$ and $p_1 >p_{\Phi}$.
In view of Remark
\ref{Rem:WeakLp} and Proposition \ref{Prop:LpCont},
\begin{equation}
\|\mathrm{Op}(a) f\|_{wL^{p_j}}
\leqslant
\|\mathrm{Op}(a) f\|_{L^{p_j}}
\leqslant C\|f\|_{L^{p_j}},
\quad f \in L^{p_j}(\rr d),
\ j=0,1.
\end{equation}
Then it follows that $\op _A(a)$ extends uniquely
to a continuous map from
$L^{p_0}(\rr d)+L^{p_1}(\rr d)$ to
$wL^{p_0}(\rr d)+wL^{p_1}(\rr d)$ (see e.{\,}g.
\cite{BerLof}).
Hence the conditions of Proposition
\ref{Prop:OrliczIntPol} are fulfilled and the
result follows.
\end{proof}

\par

By using Proposition
\ref{Prop:HormMult} instead of
Proposition \ref{Prop:LpCont} in the
previous proof
we obtain the following
extension of H{\"o}rmander's improvement
of Mihlin's Fourier multiplier theorem.
The details are left for the reader.

\par

\begin{thm}\label{Thm:OrlContHormMult}
Let $\Phi$ be a strict Young function
and
$a\in L^\infty (\rr d\setminus 0)$
be such that
\begin{equation}
\sup _{R>0}
\left (
R^{-d+2|\alpha |}\int _{A_R}
|\partial ^\alpha a(\xi )|^2\, d\xi
\right )
\end{equation}
is finite for every $\alpha \in \nn d$
with $|\alpha |\leqslant [\frac d2]+1$,
where $A_R$ is the annulus
$\sets {\xi \in \rr d}{R<|\xi |<2R}$.
Then $a(D)$ is continuous on
$L^\Phi (\rr d)$ and on $wL^\Phi (\rr d)$.
\end{thm}

\par

Finally, employing Theorem \ref{Thm:LpSGFIOcont},
we prove the following continuity result for Fourier integral operators
on $L^\Phi$-spaces. Here we let
\begin{equation}\label{Eq:LossOrliczFIO}
\lossthreshold
\equiv
(d-1)
\max\left (
    \left|\frac{1}{p_\Phi}-\frac12\right | 
    ,
    \left|\frac{1}{q_\Phi}-\frac12\right | 
    \right ) .
\end{equation}

\par

\begin{thm}\label{Thm:OrlContSGFIOs}
    Let $\Phi$ be a strict Young function,
    $\lossthreshold$ be as in \eqref{Eq:LossOrliczFIO}
    and $m,\mu  \in \mathbf R$ be such that
    %
    \begin{equation}\label{eq:orderOrlcont}
    m  \leqslant 
    -\lossthreshold
    \quad \text{and}\quad
    \mu  \leqslant 
    -\lossthreshold ,
    \end{equation}
    with strict inequalities when
    $q_\Phi <p_\Phi$.
Suppose that $\varphi \in \phf$ and $a\in S^{m,\mu}(\rr {2d})$
is such that $|\xi|\geqslant \varepsilon$, for some $\varepsilon >0$, 
on the support of $a$. Then
$\Op_{\varphi}(a)$ from  $\mascS (\rr d)$
to $\mascS '(\rr d)$ extends uniquely
to a continuous operator from $L^\Phi (\rr d)$
to itself, and from $wL^\Phi (\rr d)$
to itself.
\end{thm}


\par

\begin{rem} 
    In contrast to 
    condition \eqref{eq:sharp-thresholds} in Theorem \ref{Thm:LpSGFIOcont},
    strict inequality is required in \eqref{eq:orderOrlcont} when  $q_\Phi <p_\Phi$. 
    The sharpness of the latter
    condition will be
    investigated elsewhere.
\end{rem}

\par

\begin{proof}
First suppose that $q_\Phi =p_\Phi =p$. Then $1<p<\infty$,
$L^\Phi = L^p$, $wL^\Phi =wL^p$ and the result follows
from Remark \ref{Rem:EquivPhi}, Theorem \ref{Thm:LpSGFIOcont}
and Proposition \ref{Prop:OrliczIntPol}.

\par

Next suppose $q_\Phi <p_\Phi$. As above, by Proposition
\ref{Prop:YounFuncEquivCond} it follows
that $q_{\Phi}>1$ and $p_{\Phi}< \infty$.
Choose $p_0,p_1\in (1,\infty )$
such that $p_0 <q_{\Phi}$, $p_1 >p_{\Phi}$
and
$$
     m,\mu  < -(d-1) \left|\frac{1}{p_j}-
     \frac12\right|
    , \qquad j=0,1.
$$
In view of Remark
\ref{Rem:WeakLp} and Theorem \ref{Thm:LpSGFIOcont},
\begin{equation}
\|\mathrm{Op}_\varphi(a) f\|_{wL^{p_j}}
\leqslant
\|\mathrm{Op}_\varphi(a) f\|_{L^{p_j}}
\leqslant C\|f\|_{L^{p_j}},
\quad f \in L^{p_j}(\rr d),
\ j=0,1.
\end{equation}
By Proposition \ref{Prop:OrliczIntPol}, the claim follows, 
arguing as in the final step of the proof of Theorem \ref{Thm:OrlContPseudo}.
\end{proof}

\par

\subsection{Sobolev-Kato-Orlicz spaces, adapted global
wave-front sets, and applications to PDEs}
The next definition
is a natural one, in view of the
results in the previous subsection.

\par

\begin{defn}\label{Def:SobOrlSpcs}
Let $s,\sigma\in\mathbf R$, $\Phi$ be a Young function
and $\vartheta_{s,\sigma}(x,\xi)=
\eabs x^s\eabs \xi ^\sigma$. Then the
\emph{Sobolev-Kato-Orlicz} space $\Hphi(\rr d)$ is given by
\[
\Hphi (\rr d) \equiv \sets {f\in\mascS ^\prime(\rr d)}
{\Op(\vartheta_{s,\sigma})f\in L^\Phi(\rr d)},
\]
with topology induced by the norm
$$
\nm f{H_{s,\sigma}^\Phi}=\|\Op(\vartheta_{s,\sigma})f\|_{L^\Phi}.
$$ 
\end{defn}

\par

By straight-forward computations it follows that
$\Hphi(\rr d)$ is a Banach space. Obviously,
$H^\Phi_{0,0}(\rr d)=L^\Phi(\rr d)$.

\par

By arguing as in the established SG
calculus, the following result is
an immediate consequence of Theorems 
\ref{Thm:OrlContPseudo} and 
\ref{Thm:OrlContSGFIOs} (see, e.g., 
\cite{CJT13b,CJT16}).

\par

\begin{thm}\label{Thm:SobOrlSpcsCont}
	Let $\Phi$ be a strict Young function, 
	$\lossthreshold$ be as in \eqref{Eq:LossOrliczFIO}, 
        $l,\lambda, m, \mu$, $s, \sigma \in \mathbf R$
        be such that
    \begin{equation}\label{eq:orderOrlcont2}
    m  \leqslant 
    l-\lossthreshold
    \quad \text{and}\quad
    \mu  \leqslant 
    \lambda -\lossthreshold ,
    \end{equation}
    with strict inequalities when
    $q_\Phi <p_\Phi$.
    Suppose that $A\in \GL (d,\mathbf R)$, $\varphi \in \phf$
        and $a\in S^{m,\mu}(\rr {2d})$.
	    Then the following is true:
	    \begin{enumerate}
		\item the map $\Op _A(a)$ on
  $\mascS '(\rr d)$ restricts to a
  linear and continuous map from
  $H^\Phi _{s+m,\sigma+\mu}(\rr d)$ to $\Hphi(\rr d)$;

\vrum
  
		\item if, additionally, $|\xi|\geqslant\varepsilon$, for some
		$\varepsilon>0$, on the support of $a$, 
  then, $\Op_{\varphi}(a)$ from
  $\mascS (\rr d)$ to $\mascS '(\rr d)$
  extends uniquely  to a continuous operator
from $H^\Phi_{s+l,\sigma+\lambda}(\rr d)$ to $\Hphi(\rr d)$.
\end{enumerate}
\end{thm}

\par

\begin{proof}
    We prove only (2). The assertion (1) follows
    by similar arguments and is left for the reader. 
    Let $f\in H^\Phi_{s+l,\sigma+\lambda}$ and
    $u=\Op(\vartheta_{s+l,\sigma+\lambda})f$.
    Then, $u\in L^\Phi$, $\|f\|_{H^\Phi_{s+l,\sigma+\lambda}}=\|u\|_{L^\Phi}$, and
    \begin{align*}
        \nm {\Op_\varphi(a)f}{\Hphi}
        &=
        \nm {\Op(\vartheta_{s,\sigma})\Op_\varphi(a)
        (\Op(\vartheta_{s+l,\sigma+\lambda}))^{-1}u}{L^\Phi}
        \\[1ex]
        &\leqslant
        \nm {\Op(\vartheta_{s,\sigma})\Op_\varphi(a)
        (\Op(\vartheta_{s+l,\sigma+\lambda}))^{-1}}
        {\mathcal{L}(L^\Phi)}\cdot
        \|u\|_{L^\Phi}
        \\[1ex]
        &=
        \| \Op_\varphi(b)+\mathcal{K}\|_{\mathcal{L}(L^\Phi)}
        \cdot 
        \|f\|_{H^\Phi_{s+l,\sigma+\lambda}},
    \end{align*}
    where $\mathcal{K}$ is an operator with kernel in $\mascS (\rr {2d})$.
    In particular, $\maclK$ is continuous from $\mascS '(\rr d)$ to
    $\mascS (\rr d)$. Hence \eqref{Eq:OrlLpInclusions} gives
    $\|\mathcal{K}\|_{\mathcal{L}(L^\Phi)}<\infty$.
    By the properties of the $SG$ calculus, we have 
    $b\in S^{m-l,\mu-\lambda}$. Since $b$ is the asymptotic
    sum of an expansion involving $a$, $\vartheta_{s,\sigma}$, 
    $\vartheta_{s+l,\sigma+\lambda}$, their 
    derivatives, and suitable compositions with $\varphi^\prime_x$ and 
    $\varphi^\prime_\xi$ (see \cite{CJT16} and the references therein),
    on its support it holds $|\xi|\geqslant \varepsilon^\prime >0$,
    and the hypotheses
    of Theorem \ref{Thm:OrlContSGFIOs} are satisfied.
    We conclude that it also holds
    $\| \Op_\varphi(b)\|_{\mathcal{L}(L^\Phi)}<\infty$, and the claim follows.
\end{proof}

\begin{rem}\label{Rem:SobOrlWFS}
	We observe that $\Hphi(\rr d)$
 in Theorem \ref{Thm:SobOrlSpcsCont}
 is $SG$-admissible and that the pair
 $(H^\Phi_{s+m,\sigma+\mu}(\rr d), \Hphi(\rr d))$ is $SG$-ordered with respect to
 $(m,\mu )$ (see \cite{CJT13a,CJT13b,CJT16}
 for details and terminology).
	It is then possible to apply the theory of (global) wave-front sets developed in \cite{CJT13b} and \cite{CJT16}, choosing the
	Sobolev-Kato-Orlicz spaces from Definition \ref{Def:SobOrlSpcs} as reference spaces.
    In particular, it holds 
    \[
        f\in \Hphi(\rr d) \iff
        \WFHphi(f) \equiv \WF_{\Hphi(\rr d)} (f) =\emptyset.
    \]
\end{rem}

\par

We conclude the paper with two results about propagation of singularities with respect to the $\Hphi$ spaces. These follow, as explained in Remark \ref{Rem:SobOrlWFS},
by Theorem \ref{Thm:SobOrlSpcsCont} and the theory developed in \cite{CJT13a,CJT13b,CJT16}. 
%
%
%
%
The details are left for the reader.

\par

The first result concerns mapping properties of wave-front sets of pseudo-differential
operators with symbols in SG classes.

\par

\begin{thm}\label{Thm:SobOrlEllReg}
	Let $a\in S^{m,\mu}(\rr{2d})$, $f\in\Hphi(\rr d)$, $u\in\mascS^\prime(\rr d)$, $m,s,\mu ,\sigma\in\rr{}$,
	and consider the equation $\Op(a)u=f$. Then,
\begin{equation}\label{Eq:WFEmbeddings}
\WFHphi
(f)\subseteq\WF^\Phi_{s+m,\sigma+\mu}
(u)\subseteq
\WFHphi (f)\cup\Char(a).
\end{equation}
 %
	%
\end{thm}

\par

In addition, if $a$ in Theorem
\ref{Thm:SobOrlEllReg} is elliptic, then
$\Char(a)=\emptyset$ and
\eqref{Eq:WFEmbeddings} gives
$$
\WF^\Phi_{s+m,\sigma+\mu}(u)
=
\WFHphi(f).
$$

\par

The next two results concerns properties of solutions to hyperbolic
Cauchy problems.

\par

\begin{prop}\label{Prop:SolCP}
	Let $h\in S^{1,1}_\mathrm{cl}(\rr{2d})$ be such that $h=h_p+h_0$,
	$h_p\in S^{1,1}_\mathrm{cl}(\rr{2d})$ real-valued,
	$h_0\in S^{0,0}_\mathrm{cl}(\rr{2d})$, 
	$u_0\in H^\Phi_{s+\lossthreshold,\sigma+\lossthreshold}(\rr d)$,
	with $s,\sigma\in \mathbf R$
    and $\lossthreshold$ as in \eqref{Eq:LossOrliczFIO}. 
    Then, the solution of
	\begin{equation}\label{Eq:CP}
		\begin{cases}
			D_t u(t)=\Op(h)u(t), \; t\in[-T,T], T>0,\\
			\phantom{D_t}u(0)=u_0,
		\end{cases}
	\end{equation}
is given by
$$
u(t)=\Op_{\varphi(t)}(a(t))u_0 \, \operatorname{Mod}\,
( C^\infty([-T^\prime,T^\prime];\mascS(\rr d)) ),
\quad t\in[-T^\prime,T^\prime],
$$
for suitable families of regular phase functions $\varphi(t)$ and order 
    $(0,0)$ symbols $a(t)$. 
    \end{prop}

\par

Under the same hypotheses of Proposition \ref{Prop:SolCP}, neglecting the low frequencies by means of
a suitable cut-off function
\begin{equation}\label{Eq:SuitCuttoff}
\chi (\xi )
=
\begin{cases}
0, & |\xi | \geqslant R,
\\[1ex]
1, & |\xi | \leqslant r,
\end{cases}
\end{equation}
for the solution $u$ of the
Cauchy problem \eqref{Eq:CP}, Theorem \ref{Thm:SobOrlSpcsCont} implies that
$(1-\chi(D))u\in C^\infty ([-T^\prime,T^\prime];\Hphi(\rr d))$. Moreover, we have the following.

    \begin{thm}\label{Thm:SobOrlHypWF}
    Let $u$ be the solution in Proposition \ref{Prop:SolCP} and $\chi\in C_0^\infty(\rr d)$
    satisfy \eqref{Eq:SuitCuttoff} for some suitable $R>0$ and $r\in(0,R)$. Then
	\[
		\WFHphi((1-\chi(D))u(t)) = 
        \Psi(t)(\WF^\Phi_{s+\lossthreshold,\sigma+\lossthreshold}(u_0)),\;  
        t\in[-T^\prime,T^\prime],
	\]
	where $\Psi(t)$ is a smooth family of transformations, generated by the smooth family of phase functions $\varphi(t)$. 
\end{thm}

\par

\begin{rem}
	Local counterparts of Theorem \ref{Thm:SobOrlEllReg}, Proposition \ref{Prop:SolCP}, and Theorem \ref{Thm:SobOrlHypWF} also
	hold true when $x$ belongs to a bounded open subset of $\rr d$.
\end{rem}

\par


%
%
%
%

\end{document}